\documentclass[12pt]{article}
\usepackage{graphicx} 
\usepackage{amsmath}
\usepackage{amsthm, amssymb}
\usepackage{amsfonts}
\usepackage[dvipsnames]{xcolor}
\usepackage{tikz,tikz-network}
\usepackage{tcolorbox}
\usepackage{graphicx}
\usepackage{physics}
\usepackage{dsfont}
\usepackage{xcolor}
\usepackage{appendix}
\usepackage{stmaryrd}
\usepackage{hyperref}
\usepackage{float}
\usepackage{dsfont}
\usepackage{comment}
\usepackage{authblk}
\usepackage{stmaryrd}
\usepackage[a4paper]{geometry}
\usetikzlibrary{shapes.geometric, positioning, calc}
\usepackage{subcaption}

\newtheorem{theorem}{Theorem}[section]
\newtheorem{proposition}[theorem]{Proposition}

\newtheorem{lemma}[theorem]{Lemma}

\newtheorem{remark}[theorem]{Remark}

\renewcommand{\tr}{\mathrm{tr}}
\renewcommand{\Tr}{\mathrm{Tr}}

\newcommand{\R}{\mathbb{R}}
\newcommand{\N}{\mathbb{N}}
\newcommand{\Z}{\mathbb{Z}}

\newcommand{\C}{\mathbb{C}}
\renewcommand{\P}{\mathbb{P}}

\newcommand{\E}{\mathbb{E}}

\newcommand{\F}{\mathbb{F}}
\newcommand{\Hom}{\text{Hom}}

\newcommand{\1}{\mathds{1}}
\newcommand{\g}{\mathbf{g}}
\newcommand{\rn}{\mathbf{R}_n}
\newcommand{\PHF}{\mathbf{PHypF}}

\title{Near optimal spectral gaps for line bundles on hyperbolic three-manifolds}
\author{Michael Magee, Anna Roig-Sanchis, Joe Thomas}

\date{\today}

\begin{document}

\maketitle

\abstract{We prove that there exists a sequence of two-torsion Hermitian line bundles on closed hyperbolic three-manifolds, with volume tending to infinity, and with asymptotically optimal smallest eigenvalue of the Laplacian. The proof stems from recent advances in the program of strongly convergent unitary representations of discrete groups.}

\section{Introduction}

Let $M = \Gamma \backslash \mathbb{H}^3$ be a closed hyperbolic three-manifold. For any Hermitian line bundle $\mathcal{L}$ over $M$ with flat linear connection $d : C^\infty(M;\mathcal{L}) \to C^\infty(M;\mathcal{L}\otimes T^{*}M)$ the Laplacian operator is $\Delta_\mathcal{L}:= d^{*}d$, where $d^{*}$ is the $L^2$-adjoint of $d$ w.r.t. the Hermitian metric. We refer to this data simply as `flat line bundle'. Each such operator has canonical extension to an unbounded operator on $L^2(M;\mathcal{L})$ with discrete non-negative spectrum
\[ 0\leq \lambda_1(\mathcal{L}), \lambda_2(\mathcal{L}), \ldots. \]

We prove the following theorem.
\begin{theorem}\label{thm:main}
There exist a sequence of closed hyperbolic three-manifolds $\{M_i\}_{i=1}^\infty$ and flat Hermitian line bundles $\mathcal{L}_i\to M_i$ such that 
\[
\mathrm{vol}( M_i ) \to \infty
\] and
\[ \lambda_1(\mathcal{L}_i) \to 1 = \inf (\mathrm{spec} (\Delta_{ \mathbb{H}^3} ))\]
as $i \to \infty$.
The $\mathcal{L}_i$ can be taken to be two-torsion ($\mathcal{L}_i\otimes \mathcal{L}_i \cong 1$) and the $M_i$ to be simultaneous covers of any fixed base manifold $M_0$.
\end{theorem}
\begin{remark}The presence of the line bundles $\mathcal{L}_i$ is a concession to the fact that we expect the result to hold without them, provided we exclude the zero eigenvalue.
\end{remark}
\begin{remark}Theorem \ref{thm:main} can also be stated in the following terms: that there are a sequence of closed hyperbolic three-manifolds with fixed point free involutions and volume tending to infinity such that the bottom of the spectrum of the Laplacian on \emph{odd} functions is asymptotically optimal (one).
\end{remark}
The proof follows the following lines, based on recent developments in the theory of strongly convergent unitary representations of discrete groups (see \cite{magee2025strongconvergenceunitarypermutation, vanhandel2026strongconvergencephenomenon} for general accounts).

For $n\in \N$ let $B_n$  denote the \emph{hyperoctahedral group} of signed permutation matrices of order $n$, namely, the matrices that can be obtained from $n\times n$ permutation matrices by replacing every `1' entry by $\pm 1$; as an abstract group $B_n$ is $S_n \ltimes (\Z / 2 \Z)^n$ with the conjugation action of $S_n$ given by permutation of coordinates.

If $M = \Gamma \backslash \mathbb{H}^3$  is a hyperbolic three-manifold, and $\phi : \Gamma\to B_n$ is a homomorphism, then this data specifies a degree-$n$ covering space $M_\phi$ of $M$ and a two-torsion flat line bundle $\mathcal{L}_\phi\to M_\phi$.

We say that a sequence $\{\phi_i \in \Hom(\Gamma, B_{n_i})\}_{i=1}^\infty$,  strongly converges to the regular representation $\lambda$ of $\Gamma$ if for every $z\in \C[\Gamma]$,
\[
\lim_{i\to\infty} \| \phi_i(z) \| = \| \lambda(z) \|,
\]
where the norms on both sides are operator norms.
We prove the following.
\begin{theorem}\label{thm:strong_convergence_hyp-3-mfold}
Let $\Gamma$ be one of the following groups.
\begin{enumerate}
    \item A right-angled Artin or Coxeter group.
    \item A hyperbolic three-manifold group. 
\end{enumerate}
There exists a sequence of hyperoctahedral representations of $\Gamma$ that strongly converge to its regular representation.
\end{theorem}

We discuss the proof of this theorem momentarily. Before this, we explain the passage from Theorem \ref{thm:strong_convergence_hyp-3-mfold} (Part 2) to Theorem \ref{thm:main}. This is based on now-standard arguments about inducing strong convergence.
Let $\{ \phi_i \}_{i=1}^\infty$  denote the sequence of hyperoctahedral representations provided by the theorem.
Sections of the associated flat vector bundle $E_{\phi_i}$ to $\phi_i$  over $M$ correspond naturally to sections of the degree two line bundle $\mathcal{L}_{\phi_i}$ over $M_{\phi_i}$, where $M_{\phi_i}$  is a finite cover of $M$. 
Now \cite[Theorem 6.1]{MageeThomas} implies that $\lambda_1(E_{\phi_i}) \to 1$ as $i\to\infty$, but the Laplacian on $E_{\phi_i}$ corresponds to the Laplacian on $\mathcal{L}_{\phi_i}$ under the previously mentioned correspondence between sections.

{\bf Proof of Theorem \ref{thm:strong_convergence_hyp-3-mfold}  on strong convergence.}
Theorem \ref{thm:strong_convergence_hyp-3-mfold} Part 2 follow from Part 1 for RAAGs by the following argument (similarly to the main result of \cite{MageeThomas}): both being hyperoctahedral and strong convergence  are preserved under induction from a finite index subgroup (see Section \ref{sec: bootstrap}), and hyperbolic three-manifold groups virtually embed into a RAAG. Indeed, closed hyperbolic 3-manifold groups (by works of Kahn-Markovic \cite{KM} and Bergeron-Wise \cite[Theorem 5.3]{BW}) are word hyperbolic groups that act properly and cocompactly on a CAT(0) cube complex and these groups virtually embed into a RAAG by work of Agol \cite{Agol} and Haglund-Wise \cite{HW}. 
Non uniform hyperbolic three manifold groups also virtually embed into RAAGs by a result of Wise \cite[Thm. 14.29]{WiseBook}. 

Because this argument in general invokes the result of Agol, which is ineffective, it is good that we have nice examples of finite covolume hyperbolic RACGS in which the source of ineffectiveness in Theorem \ref{thm:main} arises only from the proof of Theorem \ref{thm:strong_convergence_hyp-3-mfold} Part 1. We describe some of these examples in Section \ref{sec: example}.

The proof of Theorem \ref{thm:strong_convergence_hyp-3-mfold} for RAAGs or RACGs follows an iteration introduced by Gao et al. in \cite{gao2026newsourcepurelyfinite}. The key lemma of this iteration is stated in Lemma \ref{lem:Gao} below.  Reading carefully the iteration of \emph{(ibid.)}, it functions by reducing the number of edges in the defining graphs of RAAGs or RACGs and as such, the base case is the case of finitely generated free groups. Here we require a new ingredient.
\begin{theorem}
\label{thm:hyp-oct-reps-thm}
For any finitely generated free group $\mathbf{F}$ there exists a sequence of hyperoctahedral representations of $\mathbf{F}$ that strongly converge to the regular representation.
\end{theorem}
When $\mathbf{F}=\Z$ this is easy (see Section \ref{subsec:Z}). When $\mathbf{F}$ is non-abelian, Theorem \ref{thm:hyp-oct-reps-thm} is established by showing that it holds for generic representations in the most natural random model (Theorem \ref{thm: free-sc} below). 
We do this using the polynomial method of Chen-Garza Vargas-Tropp-van Handel \cite{chennew} --- it is a simple modification of \emph{(ibid.)} especially considering that the needed trace statistics were already obtained in a work of the first named author and Puder \cite{Ma.Pu2021}.
As a minor remark, we knew this was possible prior to the existence of the polynomial method, using the method of Bordenave and Collins from \cite{BordenaveCollins}, but the polynomial method produces a more concise proof.

\subsection*{Acknowledgements}

M.M. This project has received
funding from the European Research Council (ERC) under the European
Union's Horizon 2020 research and innovation programme
(grant agreements No 949143 and No 101229716). This work
was funded in part by a Philip Leverhulme prize awarded by the Leverhulme Trust.
J.T. has received funding from the Leverhulme Trust through a Leverhulme Early
Career Fellowship (Grant No. ECF-2024-440). We thank Bram Petri and Ramon van Handel for useful discussions.

\section{Bootstrap from free groups to RACGs or RAAGs}
\label{sec: bootstrap}
Let $\Gamma$  denote the defining simple graph of a RACG or RAAG which will be written $G\Gamma$. 
We recall the argument of Gao et al. from \cite{gao2026newsourcepurelyfinite} that allows us to pass to the case of free groups.
Pick any  vertex $v\in \Gamma$. Let $N(v)$  denote the neighbors of $v$ in $\Gamma$ with the induced subgraph structure. Let $\Gamma'$ be the induced subgraph of $\Gamma$ on the vertices other than $v$. 
Then $G\Gamma$ splits as the amalgamated product 
\[
G\Gamma =   G\Gamma' \ast_{GN(v)} ( GN(v) \times \Z ).
\]
By a result of Green \cite[Thm. 5.3]{Green} there is a sequence $\{H_i\}^\infty_{i=1}$  of finite index subgroups of $G\Gamma'$ such that 
\[ \bigcap_{i=1}^\infty H_i = GN(v).\]
Let $K_i$ denote the largest normal subgroup of $G\Gamma'$ contained in $H_i$; each of these are finite index subgroups of $G\Gamma'$.

The following is the new contribution of Gao et al. \cite[Lemma 2.3]{gao2026newsourcepurelyfinite}.
\begin{lemma}\label{lem:Gao}
There are a sequence of group homomorphisms
\[
G\Gamma \to G\Gamma' \times \left( G\Gamma' /K_i \ast_{H_i/K_i}  ( H_i/K_i \times \Z ) \right).
\]
such that the pull backs of the regular representations of the right hand side strongly converges as $i\to\infty$ to the regular representation of $G\Gamma$.
\end{lemma}
The proof of \cite[Lemma 2.3]{gao2026newsourcepurelyfinite} was `inspired by the work of Ozawa' \cite{ozawa2026proximalityselflessnessgroupcalgebras}, and the idea of using this type of intermediate lemma for strong convergence ($G\Gamma$ is an generalized extension of centralizers of $G\Gamma'$) goes back to \cite{Lo.Ma2025}.

 We aim to prove $G\Gamma$ is $\PHF$, the induction works as follows\footnote{We require to repeat the induction of \cite{gao2026newsourcepurelyfinite} to make it precise, so the reader can understand what we have to add.}.

The following lemma is \cite[Lemma 7.1]{Lo.Ma2025} plus the easy-to-check fact that inducing a hyperoctahedral representation from a finite index subgroup gives a hyperoctahedral representation.
\begin{lemma}\label{lem:finite_index}If $G_1 \leq G_2$ is a finite index inclusion of discrete groups, and $G_1$ is $\PHF$, then $G_2$ is $\PHF$.
\end{lemma}

The base case of the induction is that the graph $\Gamma$  has no edges. If this is the case, we appeal to our Theorem \ref{thm: free-sc} to obtain that $G\Gamma$ is $\PHF$ --- this is immediate for RAAGs but note also that free Coxeter groups have free groups as finite index subgroups and hence by Lemma \ref{lem:finite_index} are also $\PHF$.

Otherwise, we appeal to induction on the number of edges. Make sure the vertex we choose to delete from $\Gamma$ to form $\Gamma'$ has an incident edge, then there are less edges in $\Gamma'$.

Here is the inductive step.
First, the group $G\Gamma' /K_i \ast_{H_i/K_i}  ( H_i/K_i \times \Z )$ is virtually free by a result of Gregorac \cite{Gregorac}, and so is $\PHF$ by Lemma \ref{lem:finite_index}.

We assume $G\Gamma'$ is $\PHF$ by the inductive hypothesis. Now, $G\Gamma'$  is exact (e.g. by Dykema \cite{Dykema}), and using e.g. Haagerup-Thorbj{\o}rnsen \cite[Thm. 9.1]{HT}, plus the fact that tensor products of hyperoctahedral representations are hyperoctahedral, we obtain that 
\[
 G\Gamma' \times \left( G\Gamma' /K_i \ast_{H_i/K_i}  ( H_i/K_i \times \Z ) \right)
\]
is $\PHF$. Now composing the homomorphisms from Lemma \ref{lem:Gao} with the homomorphisms from the above group to hyperoctahedral groups in a suitable way, we obtain that $G\Gamma$  is $\PHF$.

\section{Proof of Theorem \ref{thm:hyp-oct-reps-thm}}

\subsection{Non-abelian free groups}

Let $\F_r$ be the free group on $r$ generators with fixed generating set $\{g_1,\ldots,g_r\}$ and write $g_i=g^{-1}_{i-r}$ for $i=r+1,\ldots,2r$ and $\g=(g_1,\ldots, g_r)$. We denote by $\lambda_{\F_r}$ the regular representation of $\F_r$ which acts on $\ell^2(\F_r)$ by 
    $$\left(\lambda_{\F_r}(g)f\right)(h)=f(g^{-1}h).$$

Recall from the introduction that $B_n$ is the group of signed $n\times n$ permutation matrices. We define random hyperoctahedral representations $\rho_n$ of $\F_r$ by 
$$g_i\mapsto R_i^n,$$
where $R_i^n$ are i.i.d. uniformly random signed permutation matrices in $B_n$. 
In this section we prove the following.

\begin{theorem}
\label{thm: free-sc}
    For any $z\in \mathbb{C}[\F_r]$, 
        $$\|\rho_n(z)\|\to \|\lambda_{\F_r}(z)\|,$$
    in probability as $n\to\infty$, where on both sides the norm is the operator norm. In words, the representations $\rho_n$ strongly converge (in probability) to the regular representation $\lambda_{\F_r}$.
\end{theorem}

To any given $z\in \C[\F_r]$ we may associate a non-commutative polynomial $P\in \C\langle \mathbf{x},\mathbf{x}^*\rangle$ in the variables $\mathbf{x}=(x_1,\ldots,x_r)$ and their adjoints, such that
    $$z = P(g_1,\ldots, g_{2r}).$$
If we let $\g = (\lambda_{\F_r}(g_1),\ldots,\lambda_{\F_r}(g_r))$ and $\rn=(R_1^n,\ldots,R_r^n)=(\rho_n(g_1),\ldots,\rho_n(g_r))$ then the statement of Theorem \ref{thm: free-sc} can be restated as: for any non-commutative polynomial $P\in \C\langle \mathbf{x},\mathbf{x}^*\rangle$,
    $$\|P(\rn,\rn^*)\|\to\|P(\g,\g^*)\|,$$
in probability, where on both sides we use the operator norm on $U(n)$ and $C^*_{\mathrm{red}}(\F_r)$ respectively.

The difficulty in proving this statement is in showing that 
    \begin{align}
    \label{eq: sc-upper-bd}
    \|P(\rn,\rn^*)\|\leq\|P(\g,\g^*)\|+o_n(1),
    \end{align}
in probability. The corresponding lower bound follows from \cite[Lemma 5.14]{MageeSalle}, given that $C^*_{\mathrm{red}}(\F_r)$ has unique trace and is exact.

We can further reduce this to checking only polynomials whose coefficients are non-negative and sum to 1. This is due to a positivization trick introduced by the first named author (M.M.) and de la Salle in \cite[Proposition 6.3]{MageeSalle}, see also \cite[Lemma 2.26]{vanhandel2026strongconvergencephenomenon}. 

To this end, we directly employ the polynomial method due to Chen, Garza-Vargas, Tropp and van Handel \cite{chennew} to prove the following proposition from which the upper bound \eqref{eq: sc-upper-bd} follows identically as in \cite[Proof of Thm. 3.9]{chennew}.

We use the following notation: $\tr=\frac{1}{N}\Tr_N$ denotes the normalised trace and $\|f\|_{C^m[a,b]}=\sum_{j=0}^m \sup_{x\in [a,b]}|f^{(j)}(x)|$ is the usual $C^m$ norm restricted to an interval $[a,b]$.

\begin{proposition}
\label{prop:master_ineq}
    Fix a non-commutative polynomial $P\in  \R_{\geq 0}\langle\mathbf{x},\mathbf{x}^*\rangle$ of degree $q_0$. There exist compactly supported distributions $\nu_i$ on $C^\infty(\R)$ for every $i\in\N\cup\{0\}$ such that for every $m\geq 1$,
        \begin{align*}
            \left|\E\ \tr\left(h\circ P\left(\rn,\rn^*\right)\right)-\sum_{i=0}^{m-1}\frac{\nu_i(h)}{n^i}\right| \leq \frac{(4q_0r)^{4m}}{n^m}\left\|h\left(K\cos(\theta)\right)\right\|_{C^{4m+1}([0,2\pi])},
        \end{align*}
    where $K=\left\|P(\g,\g^*)\right\|_{C^*(\F_r)}$. Moreover, $\nu_0$ and $\nu_1$ are supported on
    $$[-\|P(\g,\g^*)\|_{C^*_{\mathrm{red}}(\F_r)},\|P(\g,\g^*)\|_{C^*_{\mathrm{red}}(\F_r)}].$$
\end{proposition}

We now proceed with the proof of Proposition \ref{prop:master_ineq}. The starting point for the polynomial method is to obtain an expansion in the case when $h$ is a polynomial. In this case, $h\circ P$ is a polynomial of degree $\deg(h) \deg(P)$ and each monomial term is a word in the free group evaluated at the $R_i^n$. The expected trace of such a word has been studied by the first named author (M.M.) and Puder in \cite{Ma.Pu2021}. 

We introduce the following notation: if $w$ is a word in the generating set of the free group, then we write $w(\rn,\rn^*)$ to be the word in the $R_i^n$ and their adjoints obtained under the map $\rho_n$. Further, let $\F_r^{\mathrm{np}}$ denote the collection of words in the free group that are not proper powers. Recall that any word $w$ can be written as $v^k$ for some $v\in \F_r^\mathrm{np}$ and $k\in\N$.

\begin{theorem}[{\cite[Lemma 1.9, Theorem 1.11]{Ma.Pu2021}}]
\label{thm: magee-puder}
For any word $w=v^k$ with $v\in\F_r^\mathrm{np}$ and $k\in\N$ we have
    \begin{align*}
        \E\ \tr\ w(\rn,\rn^*) = \begin{cases}
1 & \text{if $w=\mathrm{id}$,}\\
n^{-1}d(k/2) + O_w(n^{-2}) & \text{if $k$ is even,}\\
O_w(n^{-2}) & \text{otherwise.}
        \end{cases}
    \end{align*}
Here $d(k/2)$ is the number of divisors of $k/2$. Additionally, one may write 
$$\E\ \Tr\ w(\rn,\rn^*)=\frac{f_w(\frac{1}{n})}{g_{|w|}(\frac{1}{n})},$$ 
for any $n>|w|$, where $f_w$ and $g_{|w|}$ are polynomials of degree at most $r|w|$, depending on $w$ and $|w|$ respectively. In fact, 
    $$g_{|w|}(x) = \left(\prod_{i=1}^{|w|}\left(1-\frac{j}{n}\right)\right)^r.$$
\end{theorem}

\begin{proof}
    The first statement is exactly \cite[Thm. 1.11]{Ma.Pu2021} after unpacking the definition of $\chi_2(w)$ therein, and computing the prefactor of the leading order explicitly. The second claim is \cite[Lemma 3.1]{Ma.Pu2021} that shows that whenever $n>|w|$,
        \begin{align*}
            \E\ \Tr\ w(\rn,\rn^*) &= \sum_\Gamma \frac{n(n-1)\cdots (n-|V(\Gamma)|+1)}{\prod_{i=1}^r n(n-1)\cdots(n-|E_i(\Gamma)|+1)}\\
            &=\frac{\sum_\Gamma \frac{1}{n}^{|E(\Gamma)|-|V(\Gamma)|}\prod_{j=1}^{|w|-1}\left(1-\frac{j}{n}\right)\prod_{j=1}^r\prod_{j=|E_i(\Gamma)|}^{|w|}\left(1-\frac{j}{n}\right)}{\left(\prod_{j=1}^{|w|}\left(1-\frac{j}{n}\right)\right)^r},
        \end{align*}
    where the summation is over a finite (depending on $w$) collection of connected graphs with edges $E(\Gamma)$ labelled by generators of $\F_r$ possibly with repetition, and vertices $V(\Gamma)$ labelled from $1,\ldots, n$ without repetition. Moreover, $E_i(\Gamma)$ denotes the number of edges labelled by $g_i$ and $\sum_{i=1}^r|E_i(\Gamma)|=|E(\Gamma)|\leq |w|$. 

    The result then follows because 
        $$f_w(x) = \sum_\Gamma x^{|E(\Gamma)|-|V(\Gamma)|}\prod_{j=1}^{|w|-1}\left(1-jx\right)\prod_{j=1}^r\prod_{j=|E_i(\Gamma)|}^{|w|}\left(1-jx\right),$$
    has degree at most $r|w|$ and likewise for $g_{|w|}(x)$.
\end{proof}

The existence of an expansion as in Proposition \ref{prop:master_ineq} for polynomial $h$ now follows from the rationality of the expected trace and the bounded degrees of the polynomials in Theorem \ref{thm: magee-puder} identically to \cite[Section 6]{chennew}, giving the following.

\begin{proposition}
\label{prop: master-polys}
    Fix a non-commutative polynomial $P\in \R_{\geq 0}\langle\mathbf{x},\mathbf{x}^*\rangle$ of degree $q_0$. There exist linear functionals $\nu_i$ on the space of polynomials for every $i\in\N\cup\{0\}$ such that for every $m\geq 1$ and $h$ a polynomial of degree at most $q$,
        \begin{align*}
            \left|\E\ \tr\left(h\circ P\left(\rn,\rn^*\right)\right)-\sum_{i=0}^{m-1}\frac{\nu_i(h)}{n^i}\right| \leq \frac{(4qq_0r)^{4m}}{n^m}\left\|h\right\|_{C^0([-K,K])},
        \end{align*}
    where $K=\|P(\g,\g^*)\|_{C^*(\F_r)}$ for any $n\geq 1$. 
\end{proposition}

The extension to the existence of compactly supported distributions for which the expansion in Proposition \ref{prop:master_ineq} for smooth functions follows identically to \cite[Section 7]{chennew}. 

The only missing input for Proposition \ref{prop:master_ineq} is now the claim about the supports of $\nu_0$ and $\nu_1$. For $\nu_1$, we use the fact that a bound on the support of a compactly supported distribution can be computed through its moments in the sense that \cite[Lemma 4.9]{chennew} $\mathrm{supp}\ \nu_1\subseteq [-A,A]$ if 
    $$\limsup_{p\to\infty}|\nu_1(x^p)|^\frac{1}{p}= A.$$ 

Towards this, we will use the first part of Theorem \ref{thm: magee-puder} to obtain an expression for the moments of $\nu_1$. Let $q$ be the degree of the polynomial $P$, then we can write
    $$P(\mathbf{x},\mathbf{x}^*)=\sum_{i_1,\ldots,i_q=0}^{2r}a_{i_1,\ldots,i_q}x_{i_1}\cdots x_{i_q},$$
where we have written $x_i=x_{i-r}^*$ for $i=r+1,\ldots,2r$ for brevity.

\begin{proposition}
\label{prop: nu_1}
    For all $p\in\N$ we have
    \begin{align*}
        \nu_1(x^p) = \sum_{k=1}^{\lfloor \frac{pq}{2}\rfloor} d(k)\sum_{v\in \F_r^{\mathrm{np}}}\E_{\mathbf{I}_1,\ldots,\mathbf{I}_p}\left[\1_{g_{\mathbf{I}_1}\cdots g_{\mathbf{I}_{p}}=v^{2k}}\right],
    \end{align*}
    where each $\mathbf{I}_s=(I_{(s-1)q+1},\ldots,I_{sq})$ is an independent random vector with distribution specified by $\P(\mathbf{I}_s=(i_{(s-1)q+1},\ldots,i_{sq}))=a_{i_{(s-1)q+1},\ldots,i_{sq}}$ for each $s\in\N$.
\end{proposition}

\begin{proof}
From Proposition \ref{prop: master-polys} we have 
         $$\nu_1(x^p) = \lim_{n\to\infty} n\left(\E\ \tr(P(\rn,\rn^*)^p)-\nu_0(x^p)\right).$$
Then,
\begin{align*}
             \E\ &\tr(P(\rn,\rn^*)^p)-\nu_0(x^p) \\
             &= \sum_{i_1,\ldots,i_{pq}=0}^{2r} a_{i_1,\ldots,i_q}\cdots a_{i_{(p-1)q+1},\ldots,i_{pq}}\E\ \tr(w_{i_1,\ldots,i_{pq}}(\rn,\rn^*))-\nu_0(x^p)\\
             &=\sum_{\substack{i_1,\ldots,i_{pq}=0 \\ g_{i_1}\cdots g_{i_{pq}}\neq e}}^{2r} a_{i_1,\ldots,i_q}\cdots a_{i_{(p-1)q+1},\ldots,i_{pq}}\E\ \tr(w_{i_1,\ldots,i_{pq}}(\rn,\rn^*)),
         \end{align*}
     where we write $w_{i_1,\ldots,i_{pq}}$ for the word in the free group given by $g_{i_1}\cdots g_{i_{pq}}$. The second equality follows from the fact that
        $$\nu_0(x^p)=\lim_{n\to\infty}\E\ \tr(P(\rn,\rn^*)^p) =\langle \delta_e,\lambda(P(\g,\g^*))\delta_e\rangle=\tau(P(\g,\g^*)^p),$$
    where the first equality is due to Proposition \ref{prop: master-polys} and the second is due to Theorem \ref{thm: magee-puder}. Here $\tau$ is the trace on $C^*_{\mathrm{red}}(\F_r)$.
    Thus $\nu_0(x^p)$ is equal to the sum of the coefficients of the words $w_{i_1,\ldots,i_{pq}}$ that reduce to the identity, and this matches the expected trace for such words and yields the cancellation.

     After multiplying through by $n$ and taking the limit $n\to\infty$,  Theorem \ref{thm: magee-puder}, states that the only non-zero terms are those for which the words $w_{i_1,\ldots,i_{pq}}$ are an even power $2k$ of a non-proper power in $\F_r$ and in this case, the contribution is $d(k)$. Since $\deg P=q$ and we take the $p$th power, we must have $k\leq\lfloor\frac{pq}{2}\rfloor$. Thus,
     
     $$\nu_1(x^p) = \sum_{k=1}^{\lfloor \frac{pq}{2}\rfloor} d(k)\sum_{v\in \F_r^{\mathrm{np}}}\sum_{i_1,\ldots,i_{pq}=0}^{2r} a_{i_1,\ldots,i_q}\cdots a_{i_{(p-1)q},\ldots,i_{pq}}\1_{g_{i_1}\cdots g_{i_{pq}}=v^{2k}}.$$

     If we define the random vectors $\mathbf{I}_s$ as in the proposition statement, then it follows that
        $$\E_{\mathbf{I}_1,\ldots,\mathbf{I}_p}\left[\1_{g_{\mathbf{I}_1}\cdots g_{\mathbf{I}_{p}}=v^{2k}}\right]=\sum_{i_1,\ldots,i_{pq}=0}^{2r} a_{i_1,\ldots,i_q}\cdots a_{i_{(p-1)q},\ldots,i_{pq}}\1_{g_{i_1}\cdots g_{i_{pq}}=v^{2k}},$$
    from which the result follows.
    
\end{proof}

The proof of Proposition \ref{prop:master_ineq} now readily follows.

\begin{proof}[Proof of Proposition \ref{prop:master_ineq}]
The expansion is given by the extension of Proposition \ref{prop: master-polys} to smooth functions as remarked after the proposition. To obtain the support of $\nu_0$, recall that as in the proof of Proposition \ref{prop: nu_1}, for $h$ a polynomial we have
    $$\nu_0(h) = \tau(h\circ P),$$
and so this expression for $\nu_0$ also extends to smooth functions. In particular, this proves the support for $\nu_0$.

For $\nu_1$, we use the characterization of the distribution support in terms of the moments. In particular, from the form of $\nu_1(x^p)$ given by Proposition \ref{prop: nu_1}, the computation of
    $$\limsup_{p\to\infty}|\nu_1(x^p)|^{\frac{1}{p}}=\|P(\g,\g^*)\|_{C^*_{\mathrm{red}}(\F_r)}$$
follows in a near identical manner to \cite[Proof of Lemma 3.10]{vanhandel2026strongconvergencephenomenon}.
\end{proof}

\subsection{$\Z$}
\label{subsec:Z}

For $\Z$, we can construct an explicit sequence of permutation representations that strongly converge to the regular representation. Indeed, let $C_n$ be the $n\times n$ matrix with $1$'s just above the diagonal, $1$ in the $(n,1)$ coordinate and zeroes elsewhere; that is,
    $$C_n = \begin{pmatrix}
        0 & 1 & 0 & \cdots & 0 \\
        0 & 0 & 1 & \cdots & 0 \\
        \vdots & \ddots & \ddots & \ddots & \vdots \\
        0 & 0 & 0 & \cdots & 1\\
        1 & 0 & 0 & \cdots & 0        
    \end{pmatrix}.$$
We construct the representation $\pi_n:\Z=\langle m\rangle \to \mathrm{Sym}(n)$ by mapping $m\mapsto C_n$. The following is contained in \cite[Proof of Lemma 5.1]{Lo.Ma2025}.

\begin{proposition}
The representations $\pi_n$ converge strongly to the regular representation $\lambda_{\Z}$.
\end{proposition}

\section{Some examples}
\label{sec: example}

An example of finite covolume hyperbolic RACG can be obtained by considering the group generated by the reflections along the faces of a right-angled hyperbolic polyhedron. An interesting infinite family of these are the so-called \textit{Löbell polyhedra} $R(n)$. For every $n\geq 5$, $R(n)$ is a bounded polyhedron whose top and bottom faces are $n-$gons and whose lateral surface consists of $2n$ pentagons.
The corresponding reflection groups $G(n)$ generate a RACG, and their quotients $\mathbb{H}^3 / G(n)$ are compact hyperbolic 3-orbifolds, called Löbell orbifolds. The remainder of this section is based on \cite{surveyVesnin}, which includes a very rich description of these manifolds.

In order to obtain a closed manifold, one needs to find a small index subgroup in $G(n)$ without torsion elements. One can do so by following the general method introduced in \cite{Vesnin}, which consists of constructing a homomorphism $\varphi$ from the RACG to a finite Abelian group whose kernel is torsion free. More precisely, the Abelian group taken here is the stabiliser of a point, isomorphic to $\mathbb{Z}_2^3$. This map $\varphi$ can be described by a choice of colouring of the faces of the polyhedron. Under some condition of local linear independence (see \cite[Lemma 1]{Vesnin}) one obtains that the quotient $\mathbb{H}^3 / \text{Ker}(\varphi)$ is a closed hyperbolic 3-manifold, called Löbell manifold. This can be orientable or not; the condition for being orientable is given by \cite[Lemma 2]{Vesnin}.

Let us present now a concrete example by setting $n=5$. The polyhedron $R(5)$ corresponds to the right-angled hyperbolic dodecahedron. Its reflection group $G(5)$ has then the following presentation:
\begin{align*}
    G(5) =  \langle r_1, \ldots, r_{12}\  | \ &r_i^2 =1, \ \ i=1,\ldots,12; \\
    &r_{11}r_i = r_{11}r_i, \ r_{12}r_{5+i}=r_{5+i}r_{12}, \ i=1,\ldots,5; \\
    &r_ir_{i+1} = r_{i+1}r_i, \  i=1,\ldots,9, \ r_1r_5=r_5r_1; \\
    &r_ir_{5+i}=r_{5+i}r_i, \ i=1,\ldots,5, \ r_6r_{10}=r_{10}r_6; \\
    &r_ir_{6+i}=r_{6+i}r_i \ i=1,\ldots,4\rangle.
\end{align*}

The quotient $\mathbb{H}^3/G(5)$ is a compact hyperbolic 3-orbifold which can be visualized as the polyhedron itself but with its faces replaced by mirrors. 

In order to get an orientable closed manifold, we can describe the homomorphism $\varphi: G(5) \rightarrow \mathbb{Z}_2^3$ by the following 4-colouring (see also Figure \ref{dodecahedron}): let $\alpha, \beta, \gamma$ be three linearly independent elements of $\mathbb{Z}_2^3$, and $\delta = \alpha + \beta + \gamma$. Then $\varphi$ is defined by:
\[
\begin{aligned}
\varphi(r_{6})  &= \varphi(r_{8})  = \varphi(r_{11}) = \alpha,\\
\varphi(r_{1})  &= \varphi(r_{3})  = \varphi(r_{10}) = \beta,\\
\varphi(r_{2})  &= \varphi(r_{4})  = \varphi(r_{12}) = \gamma,\\
\varphi(r_{5})  &= \varphi(r_{7})  = \varphi(r_{9})  = \delta.
\end{aligned}
\]
By \cite[Theorem 1]{Vesnin} one has that $\mathbb{H}^3/\text{Ker}(\varphi)$ is an orientable closed hyperbolic 3-manifold. 

The colouring induces a decomposition of the set of faces into equivalence classes of those with the same colour. We say that two $\mathbb{Z}_2^3-$colourings are equivalent if some symmetry of the polyhedron takes equivalence classes relative to one colour to another. For our example, $G(5)$, there are 25 possible colourings of $R(5)$, of which only the one shown above gives an orientable manifold. Moreover, all closed manifolds corresponding to the same equivalence class are isometric, by \cite[Theorem 3.3]{Garrison-Scott}.

\begin{figure}[H]
\centering
\begin{tikzpicture}[line width=1pt, scale=0.4]
  \coordinate (A0) at (0.00000,5.00000);
  \coordinate (B0) at (-3.81928,1.24096);
  \coordinate (C0) at (-1.61473,2.22249);
  \coordinate (D0) at (-0.88612,1.21964);
  \coordinate (A1) at (-4.75528,1.54508);
  \coordinate (B1) at (-2.36044,-3.24887);
  \coordinate (C1) at (-2.61269,-0.84891);
  \coordinate (D1) at (-1.43377,-0.46586);
  \coordinate (A2) at (-2.93893,-4.04508);
  \coordinate (B2) at (2.36044,-3.24887);
  \coordinate (C2) at (0.00000,-2.74714);
  \coordinate (D2) at (0.00000,-1.50756);
  \coordinate (A3) at (2.93893,-4.04508);
  \coordinate (B3) at (3.81928,1.24096);
  \coordinate (C3) at (2.61269,-0.84891);
  \coordinate (D3) at (1.43377,-0.46586);
  \coordinate (A4) at (4.75528,1.54508);
  \coordinate (B4) at (0.00000,4.01583);
  \coordinate (C4) at (1.61473,2.22249);
  \coordinate (D4) at (0.88612,1.21964);

  \foreach \i in {0,...,4}{
    \pgfmathtruncatemacro{\inext}{mod(\i+1,5)}
    \draw (A\i) -- (A\inext);
    \draw (B\i) -- (A\inext);
    \draw (B\i) -- (C\i);
    \draw (B\i) -- (C\inext);
    \draw (C\i) -- (D\i);
    \draw (D\i) -- (D\inext);
  }

  \node[font=\footnotesize] at (0,5.9) {$12$};              
  \node[font=\footnotesize] at (-2.03786,2.80487) {$6$};
  \node[font=\footnotesize] at (-3.29732,-1.07137) {$10$};
  \node[font=\footnotesize] at (0.00000,-3.46701) {$9$};
  \node[font=\footnotesize] at (3.29732,-1.07137) {$8$};
  \node[font=\footnotesize] at (2.03786,2.80487) {$7$};
  \node[font=\footnotesize] at (-2.07332,0.67366) {$5$};
  \node[font=\footnotesize] at (-1.28138,-1.76367) {$4$};
  \node[font=\footnotesize] at (1.28138,-1.76367) {$3$};
  \node[font=\footnotesize] at (2.07332,0.67366) {$2$};
  \node[font=\footnotesize] at (0.00000,2.18002) {$1$};
  \node[font=\footnotesize] at (0,0) {$11$};                

  \fill (A0) circle (2.2pt);
  \fill (B0) circle (2.2pt);
  \fill (C0) circle (2.2pt);
  \fill (D0) circle (2.2pt);
  \fill (A1) circle (2.2pt);
  \fill (B1) circle (2.2pt);
  \fill (C1) circle (2.2pt);
  \fill (D1) circle (2.2pt);
  \fill (A2) circle (2.2pt);
  \fill (B2) circle (2.2pt);
  \fill (C2) circle (2.2pt);
  \fill (D2) circle (2.2pt);
  \fill (A3) circle (2.2pt);
  \fill (B3) circle (2.2pt);
  \fill (C3) circle (2.2pt);
  \fill (D3) circle (2.2pt);
  \fill (A4) circle (2.2pt);
  \fill (B4) circle (2.2pt);
  \fill (C4) circle (2.2pt);
  \fill (D4) circle (2.2pt);
\end{tikzpicture}
\qquad \qquad
\begin{tikzpicture}[line width=1pt, scale=0.4]

\colorlet{alphacolor}{red!30}
\colorlet{betacolor}{blue!25}
\colorlet{gammacolor}{green!25}
\colorlet{deltacolor}{yellow!35}

  \coordinate (A0) at (0.00000,5.00000);
  \coordinate (B0) at (-3.81928,1.24096);
  \coordinate (C0) at (-1.61473,2.22249);
  \coordinate (D0) at (-0.88612,1.21964);
  \coordinate (A1) at (-4.75528,1.54508);
  \coordinate (B1) at (-2.36044,-3.24887);
  \coordinate (C1) at (-2.61269,-0.84891);
  \coordinate (D1) at (-1.43377,-0.46586);
  \coordinate (A2) at (-2.93893,-4.04508);
  \coordinate (B2) at (2.36044,-3.24887);
  \coordinate (C2) at (0.00000,-2.74714);
  \coordinate (D2) at (0.00000,-1.50756);
  \coordinate (A3) at (2.93893,-4.04508);
  \coordinate (B3) at (3.81928,1.24096);
  \coordinate (C3) at (2.61269,-0.84891);
  \coordinate (D3) at (1.43377,-0.46586);
  \coordinate (A4) at (4.75528,1.54508);
  \coordinate (B4) at (0.00000,4.01583);
  \coordinate (C4) at (1.61473,2.22249);
  \coordinate (D4) at (0.88612,1.21964);

  \fill[gammacolor] (A0)--(A1)--(A2)--(A3)--(A4)--cycle;                
  \fill[alphacolor] (A0)--(A1)--(B0)--(C0)--(B4)--cycle;                
  \fill[betacolor]  (A1)--(A2)--(B1)--(C1)--(B0)--cycle;                
  \fill[deltacolor] (A2)--(A3)--(B2)--(C2)--(B1)--cycle;                
  \fill[alphacolor] (A3)--(A4)--(B3)--(C3)--(B2)--cycle;                
  \fill[deltacolor] (A4)--(A0)--(B4)--(C4)--(B3)--cycle;                
  \fill[deltacolor] (B0)--(C0)--(D0)--(D1)--(C1)--cycle;                
  \fill[gammacolor] (B1)--(C1)--(D1)--(D2)--(C2)--cycle;                
  \fill[betacolor]  (B2)--(C2)--(D2)--(D3)--(C3)--cycle;                
  \fill[gammacolor] (B3)--(C3)--(D3)--(D4)--(C4)--cycle;                
  \fill[betacolor]  (B4)--(C4)--(D4)--(D0)--(C0)--cycle;                
  \fill[alphacolor] (D0)--(D1)--(D2)--(D3)--(D4)--cycle;                

  \foreach \i in {0,...,4}{
    \pgfmathtruncatemacro{\inext}{mod(\i+1,5)}
    \draw (A\i) -- (A\inext);
    \draw (B\i) -- (A\inext);
    \draw (B\i) -- (C\i);
    \draw (B\i) -- (C\inext);
    \draw (C\i) -- (D\i);
    \draw (D\i) -- (D\inext);
  }

  \node[font=\footnotesize] at (0,5.9) {$\gamma$};              
  \node[font=\footnotesize] at (-2.03786,2.80487) {$\alpha$};
  \node[font=\footnotesize] at (-3.29732,-1.07137) {$\beta$};
  \node[font=\footnotesize] at (0.00000,-3.46701) {$\delta$};
  \node[font=\footnotesize] at (3.29732,-1.07137) {$\alpha$};
  \node[font=\footnotesize] at (2.03786,2.80487) {$\delta$};
  \node[font=\footnotesize] at (-2.07332,0.67366) {$\delta$};
  \node[font=\footnotesize] at (-1.28138,-1.76367) {$\gamma$};
  \node[font=\footnotesize] at (1.28138,-1.76367) {$\beta$};
  \node[font=\footnotesize] at (2.07332,0.67366) {$\gamma$};
  \node[font=\footnotesize] at (0.00000,2.18002) {$\beta$};
  \node[font=\footnotesize] at (0,0) {$\alpha$};                

  \fill (A0) circle (2.2pt);
  \fill (B0) circle (2.2pt);
  \fill (C0) circle (2.2pt);
  \fill (D0) circle (2.2pt);
  \fill (A1) circle (2.2pt);
  \fill (B1) circle (2.2pt);
  \fill (C1) circle (2.2pt);
  \fill (D1) circle (2.2pt);
  \fill (A2) circle (2.2pt);
  \fill (B2) circle (2.2pt);
  \fill (C2) circle (2.2pt);
  \fill (D2) circle (2.2pt);
  \fill (A3) circle (2.2pt);
  \fill (B3) circle (2.2pt);
  \fill (C3) circle (2.2pt);
  \fill (D3) circle (2.2pt);
  \fill (A4) circle (2.2pt);
  \fill (B4) circle (2.2pt);
  \fill (C4) circle (2.2pt);
  \fill (D4) circle (2.2pt);
\end{tikzpicture}
\captionsetup{width=0.75\textwidth}
\caption{{\small Faces of the dodecahedron R(5) (left) and the corresponding 4-coloring $\varphi$ described above (right).}}
\label{dodecahedron}
\end{figure}
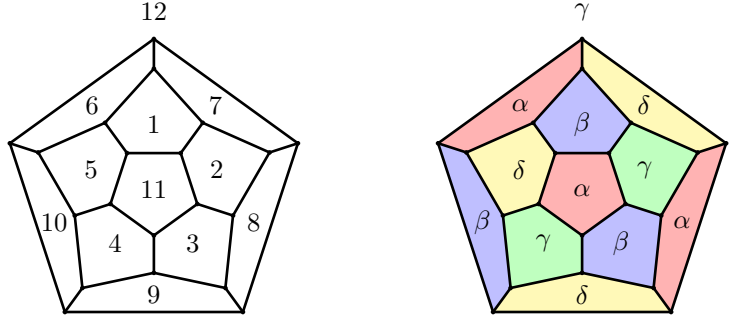

One last important property of our example is that it is an arithmetic manifold. In general, the RACGs $G(n)$ are arithmetic if and only if $n=5,6,8$ (see \cite[Lemma 3.8]{surveyVesnin}).

\bibliographystyle{amsalpha}
\bibliography{reference}

\begin{tabular}{@{}l@{}}%
Michael Magee \\
Department of Mathematical Sciences, 
Durham University, DH1 3LE Durham, UK \\
\texttt{michael.r.magee@durham.ac.uk}
\end{tabular}

\vspace{1em}

\begin{tabular}{@{}l@{}}%
Anna Roig-Sanchis \\
Laboratoire J.A. Dieudonn\'{e},
Universit\'{e} C\^{o}te d'Azur, CNRS, 06108, Nice, France \\
\texttt{anna.roig-sanchis@univ-cotedazur.fr}
\end{tabular}

\vspace{1em}

\begin{tabular}{@{}l@{}}%
Joe Thomas \\
Department of Mathematical Sciences,
Durham University, DH1 3LE Durham, UK \\
\texttt{joe.thomas@durham.ac.uk}
\end{tabular}

\end{document}